\documentclass[letterpaper, 10pt, conference]{ieeeconf}
\IEEEoverridecommandlockouts 
\usepackage{amsmath,amssymb,url}
\usepackage{graphicx,subfigure,hyperref}
\usepackage{color,comment}

\newcommand{\norm}[1]{\ensuremath{\left\| #1 \right\|}}

\newcommand{\braces}[1]{\ensuremath{\left\{ #1 \right\}}}
\newcommand{\parenth}[1]{\ensuremath{\left( #1 \right)}}

\newcommand{\refeqn}[1]{(\ref{eqn:#1})}

\newcommand{\SO}{\ensuremath{\mathsf{SO(3)}}}
\newcommand{\T}{\ensuremath{\mathsf{T}}}

\newcommand{\so}{\ensuremath{\mathfrak{so}(3)}}

\renewcommand{\Re}{\ensuremath{\mathbb{R}}}
\newcommand{\Sph}{\ensuremath{\mathsf{S}}}

\newcommand{\D}{\ensuremath{\mathbf{D}}}

\title{\LARGE \bf
Geometric Controls for a Tethered Quadrotor UAV}

\author{Taeyoung Lee%
\thanks{Taeyoung Lee, Mechanical and Aerospace Engineering, George Washington University, Washington DC 20052 {\tt tylee@gwu.edu}}
\thanks{\textsuperscript{\footnotesize\ensuremath{*}}This research has been supported in part by NSF under the grants CMMI-1243000 (transferred from 1029551), CMMI-1335008, and CNS-1337722.}
}

\newtheorem{prop}{Proposition}
\newtheorem{assump}{Assumption}

\begin{document}
\allowdisplaybreaks

\maketitle \thispagestyle{empty} \pagestyle{empty}

\begin{abstract}
This paper deals with the dynamics and controls of a quadrotor unmanned aerial vehicle that is connected to a fixed point on the ground via a tether. Tethered quadrotors have been envisaged for long-term aerial surveillance with high-speed communications. This paper presents an intrinsic form of the dynamic model of a tethered quadrotor including the coupling between deformations of the tether and the motion of the quadrotor, and it constructs geometric control systems to asymptotically stabilize the coupled dynamics of the quadrotor and the tether. The proposed global formulation of dynamics and control also avoids complexities and singularities associated with local coordinates. These are illustrated by numerical examples.
\end{abstract}

\section{Introduction}

A quadrotor unmanned aerial vehicle (UAV) consists of two pairs of counter-rotating rotors and propellers, located at the vertices of a square frame. It is capable of vertical take-off and landing (VTOL), but it does not require complex mechanical linkages, such as swash plates or teeter hinges, that commonly appear in typical helicopters. Due to its simple mechanical structure and higher thrust-to-weight ratio, it has been utilized for various applications such as mobile sensor network, and autonomous delivery systems. 

In particular, tethered quadrotors have been envisaged for aerial surveillance recently. This corresponds to a quadrotor UAV that is connected to a fixed ground point or a mobile vehicle via a tether, which can be used to for various purposes, such as supplying power to the quadrotor consistently, or high-bandwidth communication. This eliminates certain drawbacks of typical quadrotors that rely on onboard batteries and that use radio communication, namely limited flight endurance and communication bandwidth. Therefore, tethered quadrotors are particularly useful for several mission scenarios, such as prolonged aerial surveillance, communication relay, or border protection with secure, high-definition video feed~\cite{Pat15}. 

Stabilization and observer design for a tethered quadrotor have been studied in~\cite{LupDAnPIICIRS13,NicNalPIWC14,TogFraPIICRA15}. However, these results are commonly based on two simplifying assumptions that the quadrotor and the tether remain on a fixed two-dimensional plane, and the tether is always taut. As such, they may not be suitable for realistic scenarios where the quadrotor performs aggressive translational and rotational maneuvers, or the tether deforms due to the dynamic coupling with the quadrotor.

This paper is focused on eliminating such restrictive assumptions both in the dynamic model and control system design. First, we present a mathematical model of the tethered quadrotor in the three-dimensional space including deformation of the tether, which is considered as an arbitrary number of rigid links that are serially interconnected via ball joints. Next, a geometric tracking control system is designed such that the quadrotor can asymptotically follow a given desired trajectory in the three-dimensional space while controlling the tension along the tether, assuming that the tether is taut. Numerical simulations illustrate that the tension should be sufficiently large to prevent lateral vibrations of the tether, when applied to the flexible tether model. This motivates the development of another control system to stabilize both the quadrotor and the tether simultaneously, while incorporating the deformable dynamics of the flexible tether explicitly. 

Another distinct feature is that the equations of motion and the control systems are developed directly on the nonlinear configuration manifold in a coordinate-free fashion. This yields remarkably compact expressions for the dynamic model and controllers, compared with those based on local coordinates that often require symbolic computational tools due to complexity of multibody systems. Furthermore, singularities of local parameterization are completely avoided.

In short, the main contributions of this paper are summarized as (i) a global dynamic model for a tethered quadrotor with flexible tether, (ii) a geometric tracking control system for the quadrotor maneuvers in the three-dimensional space and the tension when the tether is taut, and (iii) a geometric control system for the quadrotor and the flexible tether. The second item can be considered as a new contribution independently, and the dynamics and control including flexibility of tether have been unprecedented.

\begin{figure}
\setlength{\unitlength}{0.09\columnwidth}
\centerline{\footnotesize\selectfont
\begin{picture}(5,5.5)(0,0)
\put(0,0){\includegraphics[width=0.45\columnwidth]{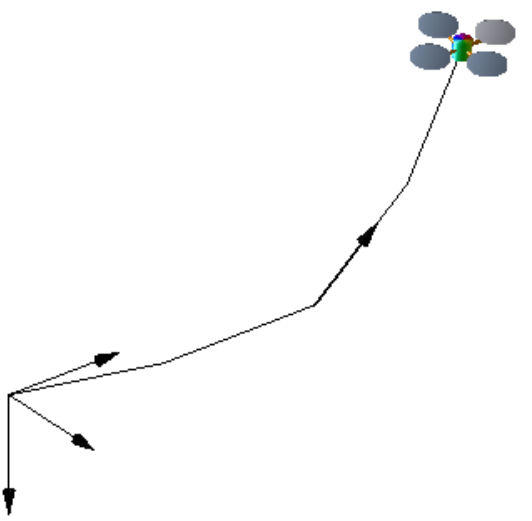}}
\put(-0.3,0){$\vec e_3$}
\put(0.8,1.7){$\vec e_1$}
\put(0.6,0.4){$\vec e_2$}
\put(3.6,2.4){$q_i\in\Sph^2$}
\put(3.9,3.1){$x_i$}
\put(2.95,1.8){$x_{i-1}$}
\put(2.6,2.8){$m_i,l_i$}
\put(3.1,4.5){$m,J$}
\put(5.0,4.2){$R\in\SO$}
\end{picture}
}
\caption{Tethered quadrotor UAV: flexible tether is modeled as a serial chain of $n$ links, and the configuration manifold is $\mathsf{Q}=(\Sph^2)^n\times \SO$.}
\end{figure}
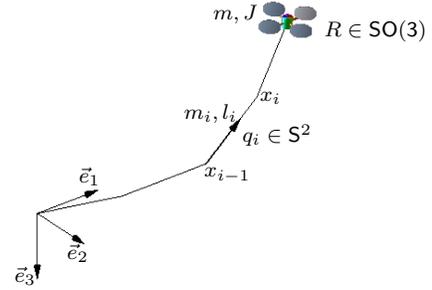

\section{Dynamics of  a Tethered Quadrotor UAV}

Consider a quadrotor UAV that is connected to a fixed point on the ground via a tether. Define an inertial frame whose origin is located at the pivot point where tether is attached to the ground. The third axis of the inertial frame is pointing downward along the direction of gravity. Define a body-fixed frame whose origin is located at the mass center of the quadrotor (see Figure 1). 

We approximate the tether by an arbitrary number, namely $n$, of rigid links that are serially connected by ball joints. The links are indexed from the ground to the quadrotor in ascending order, i.e., the first link is attached to the ground, and the $n$-th link is attached to the quadrotor. Let the mass and the length of the $i$-th link be $m_i,l_i\in\Re$, respectively, where it is assumed that the mass is uniformly distributed along each link. Throughout this paper, the subscript $i$ is considered as an element of $\{1,\ldots,n\}$. 

The direction of the $i$-th link toward the quadrotor is denoted by the unit-vector $q_i\in\Sph^2=\{q\in\Re^3\,|\, \|q\|=1\}$. The location of the outward end of the $i$-th link, namely $x_i\in\Re^3$ is given by
\begin{align}
x_i =\sum_{j=1}^i l_j q_j,\label{eqn:xi}
\end{align}
and therefore, the $i$-th link connects $x_{i-1}$ with $x_{i}$ assuming $x_0=0_{3\times 1}$.
The $n$-th link is attached to the mass center of the quadrotor such that the location of the quadrotor corresponds to $x\triangleq x_n=\sum_{j=1}^n l_jq_j\in\Re^3$. Let $R\in\SO=\{R\in\Re^{3\times 3}\,|\, R^TR=I,\,\mathrm{det}[R]=1\}$ be the rotation matrix describing the attitude of the quadrotor, and it represents the linear transformation of the representation of a vector from the body-fixed frame to the inertial frame. Therefore, the configuration manifold of the presented tethered quadrotor is $\mathsf{Q}=(\Sph^2)^n\times\SO$.

The kinematics equations are given by
\begin{gather}
\dot q_i =\omega_i \times q_i,\label{eqn:dotqi}\\
\dot R = R\hat\Omega,\label{eqn:dotR}
\end{gather}
where $\omega_i\in\Re^3$ is the angular velocity of the $i$-th link represented with respect to the inertial frame. Without loss of generality, it is assumed that $q_i\cdot\omega_i=0$, i.e., $\omega_i$ is perpendicular to $q_i$. The standard dot product is denoted by $x\cdot y = x^T y$ for any $x,y\in\Re^n$ in this paper. The vector $\Omega\in\Re^3$ corresponds to the angular velocity of the quadrotor represented with respect to the body-fixed frame, and the \textit{hat map} $\hat\cdot : \Re^3\rightarrow \so=\{S\in\Re^{3\times 3}\,|\, S^T=-S\}$ is defined such that $\hat x y =x\times y$ for any $x,y\in\Re^3$. The inverse of the hat map is denoted by the \textit{vee map} $\vee:\so\rightarrow\Re^3$.

The dynamic model of the quadrotor is identical to~\cite{LeeLeoPICDC10}. The mass and the inertia matrix of the quadrotor are denoted by $m\in\Re$ and $J\in\Re^{3\times 3}$, respectively. It generates a thrust $u\in\Re^3$ given by
\begin{align}
u = -f R e_3,\label{eqn:u0}
\end{align}
with respect to the inertial frame, where $f\in\Re$ is the total thrust magnitude and $e_3=[0,0,1]^T\in\Re^3$. It also generates a moment $M\in\Re^3$ with respect to its body-fixed frame. The control input of the presented tethered quadrotor is $(f,M)$.

\subsection{Euler--Lagrange equations}

We derive a global form of the equations of motion for the tethered quadrotor via Lagrangian mechanics. The material points on the $i$-th link are parameterized as $x_{i-1} +\zeta_i q_i$ for $\zeta_i\in[0,l_i]$, and the mass of the infinitesimal element $d\zeta_i$ corresponds to $\frac{m_i}{l_i}d\zeta_i$.  Thus, the kinetic energy of the $i$-th link can be written as
\begin{align*}
\mathcal{T}_i & =\int_{0}^{l_i}\frac{1}{2} \frac{m_i}{l_i}\|\dot x_{i-1}+\zeta_i \dot q_i\|^2 d\zeta \\
&=\frac{1}{2}m_i \|\dot x_{i-1}\|^2 +\frac{1}{2}m_il_i \dot x_{i-1}\cdot \dot q_i + \frac{1}{6}m_il_i^2 \|\dot q_i\|^2. 
\end{align*}
The kinetic energy of the quadrotor is composed of the translational kinetic energy and the rotational kinetic energy,
\begin{align*}
\mathcal{T}_{quad} = \frac{1}{2}m\|\dot x_n\|^2 + \frac{1}{2}\Omega\cdot J\Omega.
\end{align*}
The total kinetic energy is $\mathcal{T}=\mathcal{T}_{quad}+\sum_{i=1}^n \mathcal{T}_i$. From \refeqn{xi}, we have $\dot x_{i} = \sum_{j=1}^{i} l_j\dot q_i$. Substituting this and rearranging, the total kinetic energy can be written as
\begin{align}
\mathcal{T} = \frac{1}{2}\sum_{i,j=1}^n M_{ij}\dot q_i\cdot \dot q_j+\frac{1}{2}\Omega\cdot J\Omega,\label{eqn:KE}
\end{align}
where the fixed inertia terms $M_{ij}\in\Re$ are defined as
\begin{align*}
M_{ii} &= \parenth{m + \frac{1}{3}m_i} l_i^2 + \sum_{p=i+1}^n m_p l_i^2,
\end{align*}
for any $1\leq i\leq n$, and the off-diagonal terms are defined for $1\leq j < i \leq n$ as
\begin{align*}
M_{ij} = M_{ji} = \parenth{m+ \frac{1}{2}m_i}l_il_j + \sum_{p=i+1}^n m_p l_i l_j.
\end{align*}


Next, the gravitational potential energy of the $i$-th link and the quadrotor are given by
\begin{align*}
\mathcal{U}_i &= -m_i g e_3 \cdot (x_{i-1}+\frac{1}{2} l_i q_i),\\
\mathcal{U}_{quad} & = -m g e_3 \cdot x_n.
\end{align*}
Therefore, the total gravitational potential is $\mathcal{U}=\mathcal{U}_{quad}+\sum_{i=1}^n \mathcal{U}_i$, and it can be written as
\begin{align}
\mathcal{U} & = -ge_3 \cdot \braces{\sum_{i=1}^n m_i\parenth{ \sum_{j=1}^{i-1}l_jq_j + \frac{1}{2}l_i q_i}+\sum_{k=1}^n m l_k q_k}\nonumber \\
& = -g e_3 \cdot \sum_{i=1}^n M_{g_i}l_i q_i,\label{eqn:U}
\end{align}
where $M_{g_i}\in\Re$ is defined as
\begin{align*}
M_{g_i} = m+\frac{1}{2}m_i +\sum_{p=i+1}^n m_p.
\end{align*}
The Lagrangian is given by $\mathcal{L}=\mathcal{T}-\mathcal{U}$ from \refeqn{KE} and \refeqn{U}.

A coordinate-free form of Lagrangian mechanics on the two-sphere $\Sph^2$ and the special orthogonal group $\SO$ for various multibody systems has been studied in~\cite{Lee08,LeeLeoIJNME09}. The key idea is representing the infinitesimal variation of $q_i\in\Sph^2$ in terms of the exponential map:
\begin{align}
\delta q_i = \frac{d}{d\epsilon}\bigg|_{\epsilon = 0} \exp (\epsilon \hat\xi_i) q_i = \xi_i\times q_i,\label{eqn:delqi}
\end{align}
for a vector $\xi_i\in\Re^3$ with $\xi_i\cdot q_i=0$. This guarantees that the infinitesimal variation is at the correct tangent space of the two-sphere, i.e., $\delta q_i\in\T_{q_i}\Sph^2$. Similarly, the variation of $R_i$ is given by $\delta R_i  = R_i\hat\eta_i$ for $\eta_i\in\Re^3$. 

By using these expressions, the equations of motion can be obtained from Hamilton's principle as follows,
\begin{gather}
J\dot\Omega + \Omega\times J\Omega = M,\label{eqn:Wdot}\\
M_{ii}\ddot q_{i} -\hat q_i^2\sum_{\substack{j=1\\j\neq i}}^n M_{ij}\ddot q_j + M_{ii}\|\dot q_i\|^2 q_i +\hat q_i^2 M_{g_i}l_i g e_3 = -\hat q_i^2l_i u,\label{eqn:qiddot}
\end{gather}
(see Appendix A). This can be rearranged as
\begin{align}
\mathcal{M}(q) \begin{bmatrix} \ddot q_1 \\ \ddot q_2 \\ \vdots \\ \ddot q_n\end{bmatrix}
+\begin{bmatrix} \mathcal{G}_1 (q,\dot q)\\
\mathcal{G}_2 (q,\dot q)\\ \vdots \\
\mathcal{G}_n (q,\dot q)\end{bmatrix}
=
\begin{bmatrix}
-l_1\hat q_1^2 u\\
-l_2\hat q_2^2 u\\
\vdots\\
-l_n\hat q_n^2 u
\end{bmatrix},\label{eqn:qiddot_M}
\end{align}
where $\mathcal{M}(q)\in\Re^{3n\times 3n}$, $\mathcal{G}_i(q,\dot q)\in\Re^3$ are defined 
as
\begin{align}
\mathcal{M}(q)& =\begin{bmatrix} M_{11}I & -\hat q_1^2 M_{12} & \cdots & -\hat q_1^2 M_{1n}\\
-\hat q_2^2 M_{21} & M_{22}I & \cdots & -\hat q_2^2 M_{2n}\\
\vdots & \vdots & & \vdots\\
-\hat q_n^2 M_{n1} & -\hat q_n^2 M_{n2} & \cdots & M_{nn}I
\end{bmatrix},\label{eqn:MMq}\\
\mathcal{G}_i(q,\dot q) & = M_{ii}\norm{\dot{q}_i}^2 q_i + \hat q_i^2 M_{g_i} l_i g e_3.\label{eqn:GGi}
\end{align}
%
Alternatively, the equation \refeqn{qiddot} can be rewritten in terms of the angular velocity as
\begin{gather}
M_{ii}\dot\omega_i -\hat q_i \sum_{\substack{j=1\\j\neq i}}^n M_{ij}(\hat q_j\dot\omega_j +\|\omega_j\|^2 q_j)-\hat q_i M_{g_i}l_i ge_3 = l_i\hat q_i u.\label{eqn:widot}
\end{gather}
Together with the kinematics equations \refeqn{dotqi} and \refeqn{dotR}, these describe the dynamics of the tethered quadrotor.

\section{Control System Design for Taut Tether}

Designing a control system for the tethered quadrotor described by \refeqn{Wdot} and \refeqn{widot} is challenging as it is highly underactuated: there are $2n+3$ degrees of freedom but only $4$ independent control inputs. In this section, we first design a control system for the special case when there is a single link, i.e., $n=1$, assuming that the cable is always taut. Instead, the tension along the tether is also controlled such that the tether is stretched even if the deformation of the tether is included in the dynamic model. The deformation of tether will be incorporated later at Section IV. Throughout this section, the subscript $1$ is removed for brevity, i.e., $q=q_1$ and $\omega=\omega_1$.

\subsection{Problem Formulation}

When $n=1$, the equation of motion \refeqn{widot} reduces to
\begin{gather}
\dot\omega -\alpha\hat q e_3 = \beta\hat q u,\label{eqn:wdot0}
\end{gather}
where the constants $\alpha,\beta$ are given by
\begin{align*}
\alpha=\frac{M_{g_1}l_1}{M_{11}}g= \frac{m+\frac{m_1}{2}}{(m+\frac{m_1}{3})l_1}g,\quad
\beta=\frac{l_1}{M_{11}}=\frac{1}{(m+m_1/3)l_1}.
\end{align*}

Next, we find the expression of the tension along the tether. Since the location of the quadrotor is given by $x=lq$, using \refeqn{dotqi} and \refeqn{wdot0}, its acceleration can be written as 
\begin{align*}
\ddot x = l \ddot q = l (-\hat q\dot\omega -\|\omega\|^2 q)
=l(-\alpha\hat q^2 e_3 -\beta\hat q^2 u -\|\omega\|^2 q).
\end{align*}
Let $\lambda\in\Re^3$ be the internal force exerted by the tether on the quadrotor. Considering the free-body diagram of the quadrotor excluding the link, from Newton's second law, $m\ddot x = u + mge_3 + \lambda$. The tension along the link, namely $T\in\Re$ corresponds to the component of the negative internal force $-\lambda$ along the direction of the link $q$, i.e., $T=-\lambda\cdot q$. Note that it is defined such that a positive tension $T$ implies that the tethered is being stretched. By combining the above two equations, 
\begin{align*}
T & = q\cdot (u+ mg e_3-m\ddot x)
= q\cdot (u+ g e_3)+ml\|\omega\|^2.
\end{align*}

In short, the equations of motion for the tethered quadrotor and the tension when $n=1$ are given by \refeqn{Wdot}, and
\begin{gather}
\dot\omega -\alpha\hat q e_3 = \beta\hat q u^\perp,\label{eqn:wdot}\\
T = mg q\cdot e_3 + ml\|\omega\|^2 + q\cdot u^\parallel,\label{eqn:T}
\end{gather}
where $u^\perp,u^\parallel\in\Re^3$ denote the component of $u$ that is perpendicular to $u$, and the other component that is parallel to $u$, respectively, given by
\begin{align}
u^\perp = (I_{3\times 3}-qq^T) u = -\hat q^2 u,\label{eqn:uperp}\\
u^\parallel = qq^T u = (I_{3\times 3}+\hat q^2) u.\label{eqn:uparalle}
\end{align}

A tracking control problem for the tethered quadrotor is formulated as follows. Suppose that a smooth desired trajectory of the direction of the link, namely $q_d(t):\Re\rightarrow\Sph^2$, and the desired tension $T_d(t):\Re\rightarrow\Re$ are given. The desired direction satisfies
\begin{align}
\dot q_d(t) = \omega_d(t) \times q_d(t),
\end{align}
for the corresponding desired angular velocity $\omega_d(t)\in\Re^3$ satisfying $\omega_d(t)\cdot q_d(t)=0$. We wish to design the control input of the quadrotor $(f,M)$ such that this desired trajectory becomes an asymptotically stable equilibrium of the controlled system. 

\subsection{Simplified Dynamic Model $(n=1)$}

The presented quadrotor is underactuated since the total thrust is always parallel to its third body-fixed axis. This can be directly observed from the expression of the total thrust given by $u=-fRe_3$. The magnitude $f$ of the total thrust and the total control moment $M$ are arbitrary. To overcome this, we first consider a simplified dynamic model where the quadrotor may generate the total thrust along any direction. This is equivalent to designing a desired total thrust $u$ based on \refeqn{wdot} and \refeqn{T} without considering its attitude dynamics \refeqn{Wdot}. This is possible as the attitude dynamics does not directly appear in the dynamics of the link at \refeqn{wdot}. The effects of the attitude dynamics will be incorporated later.

The equations of motion for the link and the tension given by \refeqn{wdot} and \refeqn{T} have the following structure: the link dynamics is controlled by the perpendicular component of the control input $u^\perp$, and the tension is controlled by the parallel component of the control input $u^\parallel$. Therefore, $u^\perp$ is designed such that the link asymptotically follows its desired direction, and $u^\parallel$ is designed for the desired tension. The resulting complete control input is obtained by combining them together.

First, we design the parallel component. Since the tension is an algebraic function of $u^\parallel$ at \refeqn{T}, it is designed as
\begin{align}
u^\parallel = (T_d - mgq\cdot e_3 - ml\|\omega\|^2) q,\label{eqn:uprl1}
\end{align}
such that the resulting tension is identical to $T_d$ always.

Next, we design $u^\perp$ for the link dynamics \refeqn{dotqi} and \refeqn{wdot} such that $q\rightarrow q_{d}$ as $t\rightarrow\infty$. Control systems for the unit-vectors on the two-sphere have been studied in~\cite{BulLew05,Wu12}. In this paper, we adopt the control system developed in terms of the angular velocity in~\cite{Wu12}. Define the tracking error variables, namely $e_{q},e_{\omega}\in\Re^3$ as 
\begin{align*}
e_{q}  = q_{d}\times q,\quad
e_{\omega}  = \omega + \hat q^2\omega_{d}.
\end{align*}
For positive constants $k_{q},k_{\omega}\in\Re$, the normal component of the control input is chosen as
\begin{align}
u^\perp & = -\frac{1}{\beta}\hat q \{-k_q e_{q} -k_{\omega}e_{\omega} -(q\cdot\omega_{d})\dot q -\hat q^2\dot\omega_d-\alpha \hat q e_3\}.\label{eqn:uperp1}
\end{align}
Note that the expression of $u^\perp$ is perpendicular to $q$ by definition. Substituting \refeqn{uperp1} into \refeqn{wdot}, and rearranging it with the facts that the matrix $-\hat q_i^2$ corresponds to the orthogonal projection to the plane normal to $q_i$ and $\hat q_i^3=-\hat q_i$, 
\begin{align}
\dot\omega & = -k_q e_{q} -k_{\omega}e_{\omega} -(q\cdot\omega_{d})\dot q -\hat q^2\dot\omega_d.\label{eqn:dotwu}
\end{align}
In short, the control input is given by
\begin{align}
u = u^\parallel + u^\perp,\label{eqn:u}
\end{align}
from \refeqn{uprl1} and \refeqn{uperp1}, for the simplified dynamic model.

\begin{prop}
Consider the simplified dynamic model described by \refeqn{dotqi}, \refeqn{wdot}, and \refeqn{T}, where $n=1$ and $u$ can be arbitrarily selected. The control input is designed as \refeqn{u}. Then, the zero equilibrium of the tracking error, $(e_q,e_\omega)=(0,0)$ is exponentially stable, and the tension is identical to its desired value, i.e, $T(t)=T_d(t)$ for any $t$. 
\end{prop}
\begin{proof}
See Appendix B.
\end{proof}

\subsection{Full Dynamic Model $(n=1)$}

The above control system for a simplified dynamics model is generalized to the full dynamics model that includes the attitude dynamics \refeqn{dotR}, \refeqn{Wdot} of the quadrotor. The control force of the full dynamic model is given by $-f R e_3$. Here, the attitude of each quadrotors is controlled such that the direction of its third body-fixed axis, $Re_3$ becomes parallel with $-u$ given at \refeqn{u}.

The corresponding attitude controller is similar with~\cite{LeeLeoPICDC10,LeeSrePICDC13}. The desired direction of the third body-fixed axis is
\begin{align}
b_{3_c} = -\frac{u}{\|u\|}.\label{eqn:b3i}
\end{align}
There is an additional one-dimensional degree of freedom in the desired attitude, corresponding to rotation about $b_{3_c}$. A desired direction of the first body-fixed axis, $b_{1_{d}}(t)\in\Sph^2$ is introduced to resolve it. The resulting desired attitude is
\begin{align}
R_{c} = \begin{bmatrix}
-\dfrac{(\hat b_{3_c})^2 b_{1_d}}{\|(\hat b_{3_c})^2 b_{1_d}\|},&
 \dfrac{\hat b_{3_c} b_{1_d}}{\|\hat b_{3_c} b_{1_d}\|}, & b_{3_c}\end{bmatrix},\label{eqn:Rc}
\end{align}
and the desired angular velocity is obtained by $\Omega_{c}=(R_{c}^T \dot R_{c})^\vee\in\Re^3$. Define the error variables for the attitude dynamics as
\begin{align*}
e_{R}  = \frac{1}{2} (R_{c}^TR - R^T R_{c})^\vee, \quad 
e_{\Omega} = \Omega - R^T R_{c} \Omega_{c}.
\end{align*}
The thrust magnitude and the moment vector of quadrotors are chosen as 
\begin{align}
f  & = -u\cdot Re_3,\label{eqn:fi}\\
M & = -\frac{k_R}{\epsilon^2} e_{R} -\frac{k_{\Omega}}{\epsilon} e_{\Omega} +\Omega\times J\Omega\nonumber\\
&\quad -J(\hat\Omega R^T R_{c} \Omega_{c} - R^T R_{c}\dot\Omega_{c}),\label{eqn:Mi}
\end{align}
where $\epsilon,k_R,k_\Omega$ are positive constants~\cite{LeeLeoPICDC10}.

\begin{prop}\label{prop:FDM}
Consider the full dynamics model defined by \refeqn{dotqi}, \refeqn{dotR}, \refeqn{Wdot}, \refeqn{wdot}, and \refeqn{T}, where $n=1$. Control inputs $(f,M)$ are designed as \refeqn{fi} and \refeqn{Mi}, where the desired control force $u$ is given by \refeqn{u}. Then, there exist $\epsilon^* >0$, such that for all $\epsilon < \epsilon^*$, the zero equilibrium of the tracking errors $(e_{q}, e_{\omega},e_{R},e_{\Omega})$ is exponentially stable, and $\lim_{t\rightarrow\infty} T(t)=T_d(t)$.
\end{prop}

\begin{proof}
See Appendix C.
\end{proof}

\subsection{Numerical Examples}\label{sec:NE1}

Compared with the prior results where the motion of the quadrotor and the tether is restricted to a two-dimensional plane, the proposed control system guarantees that the quadrotor and the tether asymptotically follows a desired trajectory in the three-dimensional space, while controlling the tension along the tether. These are illustrated by a numerical example as follows.

Properties of the quadrotor are chosen as $m=0.755\,\mathrm{kg}$ and $J=\mathrm{diag}[0.0043,0.0043,0.0103]\,\mathrm{kgm^2}$. The total mass and the total length of the tether are $0.3\,\mathrm{kg}$ and $5\,\mathrm{m}$, respectively. Initially, the tether is aligned along a horizontal direction with zero angular velocity, i.e., $q(0)=e_1, \omega(0)=0$. The desired trajectory is chosen such that the quadrotor follows a figure-eight curve on the sphere, 
\begin{gather*}
q_d(t)=[\cos\theta\cos\phi,\, \sin\theta,\, -\cos\theta\sin\phi]^T,\\
\theta(t)=\frac{\pi}{6}\sin 0.2\pi t,\quad \phi(t)=\frac{\pi}{18}\sin 0.4\pi t + \frac{\pi}{2},
\end{gather*}
and the desired tension is $T_d=5\,\mathrm{N}$. The corresponding simulation results are presented at Figure \ref{fig:1}, where it is shown that the tracking errors converge to zero. 

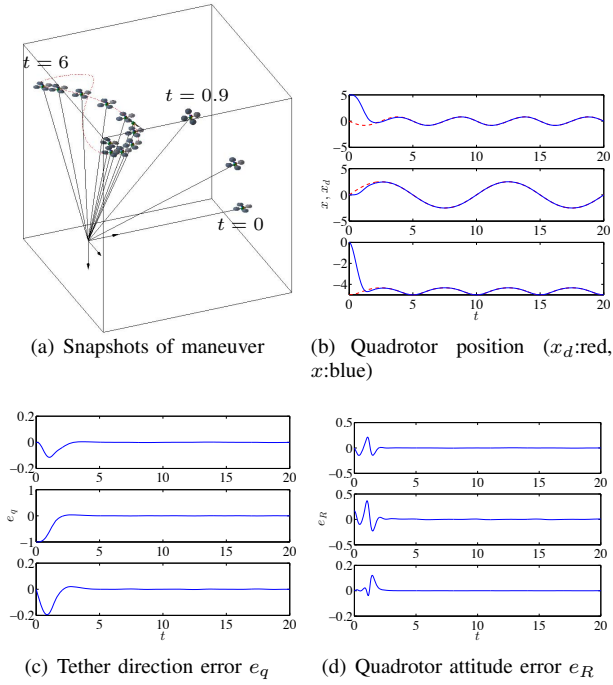
\begin{figure}
\centerline{
	\subfigure[Snapshots of maneuver]{
		\setlength{\unitlength}{0.1\columnwidth}\scriptsize
		\begin{picture}(4.6,4.9)(0,0.6)
		\put(0,0){\includegraphics[width=0.45\columnwidth]{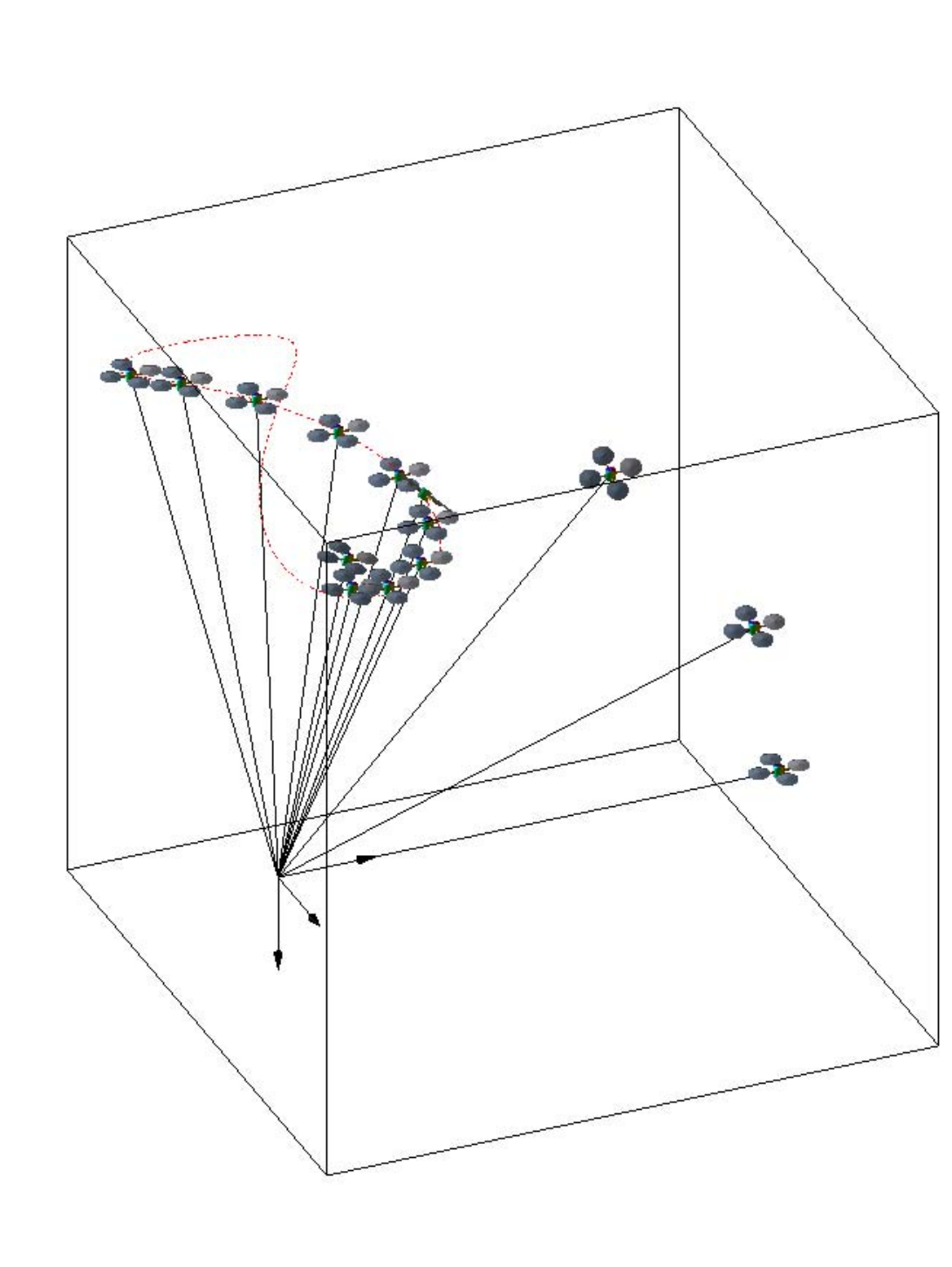}}
		\put(3.3,2){$t=0$}
		\put(2.5,4){$t=0.9$}
		\put(0.3,4.5){$t=6$}
		\end{picture}}
	\subfigure[Quadrotor position ($x_d$:red, $x$:blue)]{
		\includegraphics[width=0.45\columnwidth]{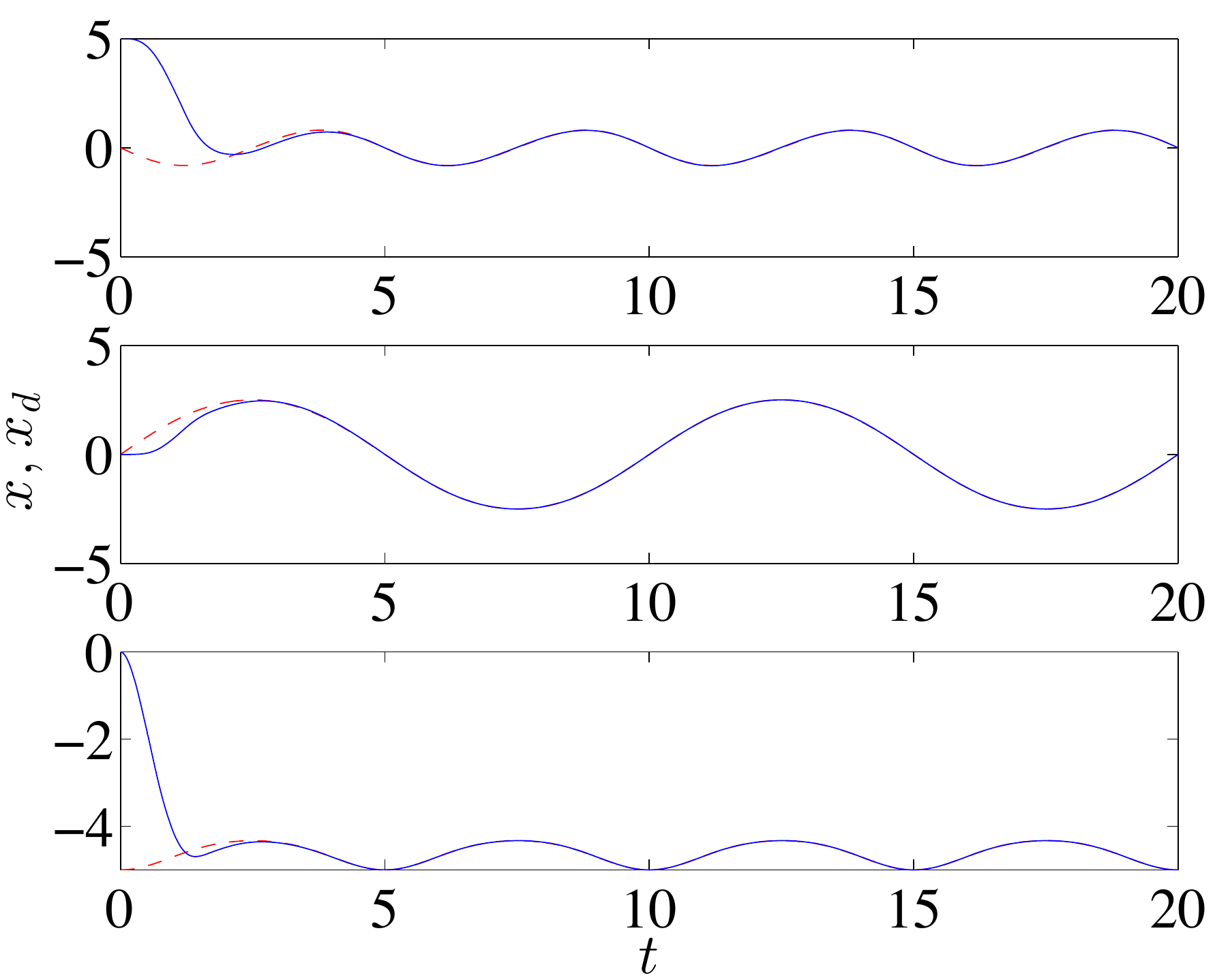}}
}		
\centerline{
	\subfigure[Tether direction error $e_q$]{
		\includegraphics[width=0.45\columnwidth]{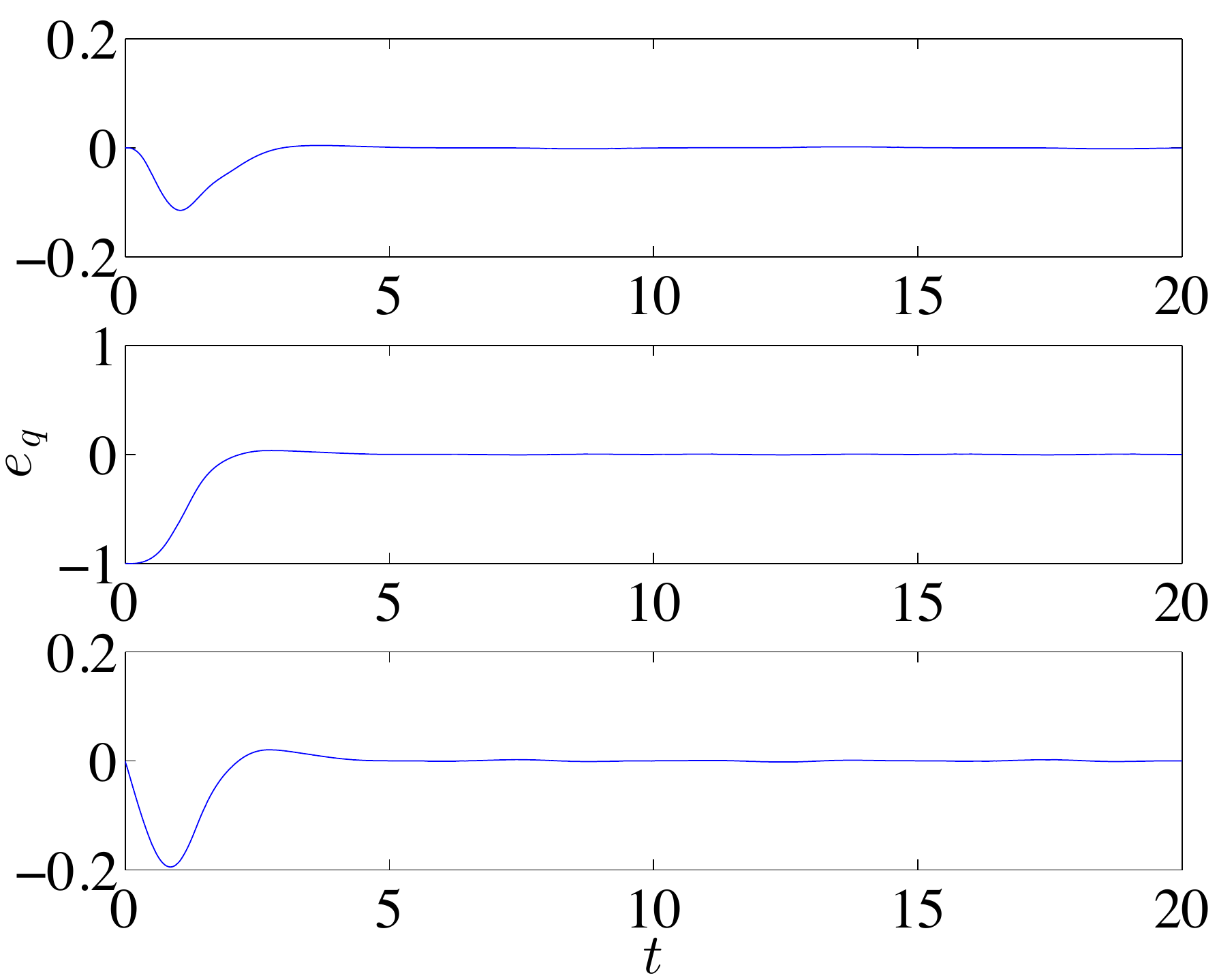}}
	\subfigure[Quadrotor attitude error $e_R$]{
		\includegraphics[width=0.45\columnwidth]{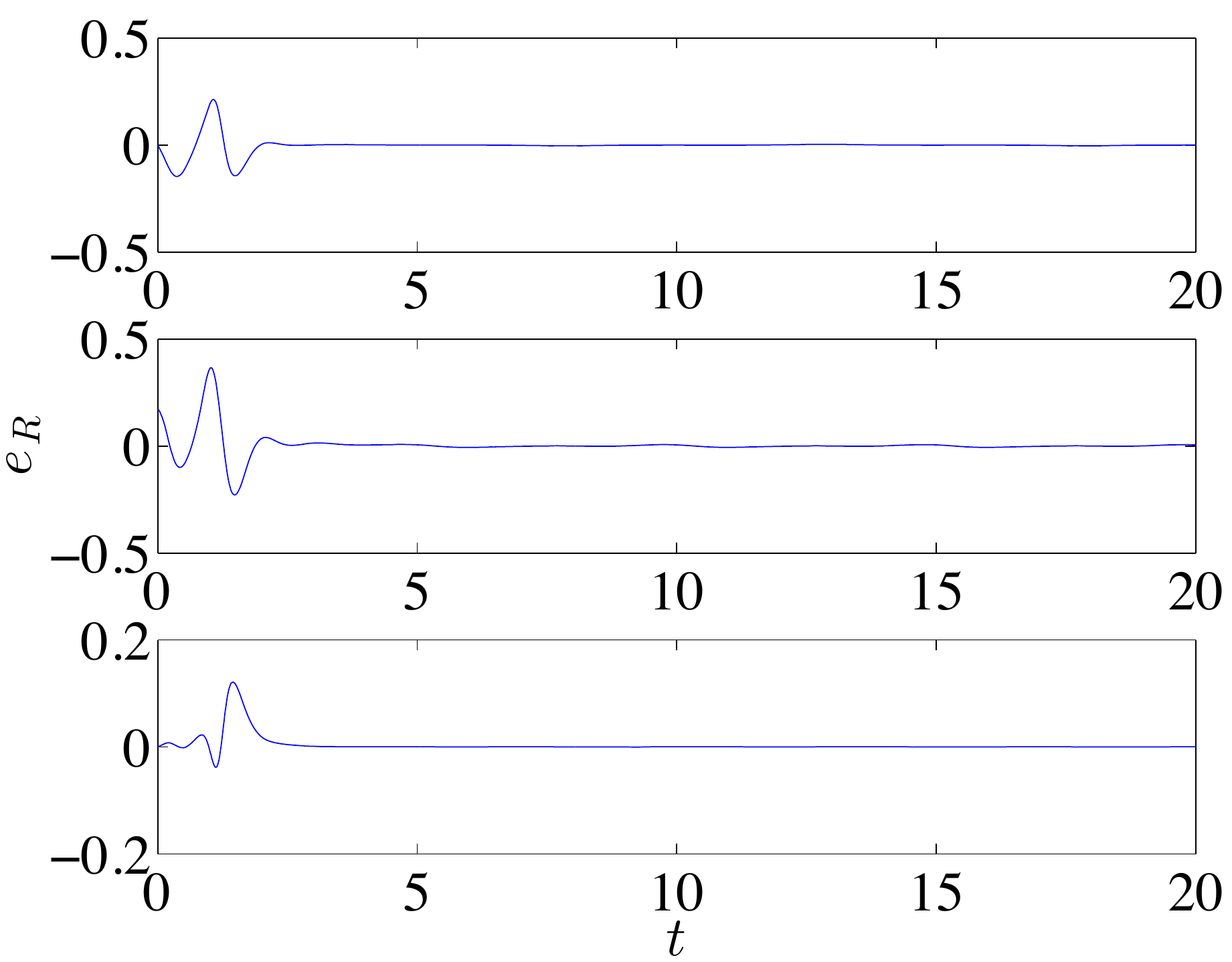}}
}
\caption{Numerical results for $n=1$ (animation available at \url{http://fdcl.seas.gwu.edu/CDC15_Fig2.mov})}\label{fig:1}
\end{figure}

Next, we apply the presented control system developed for $n=1$ into the dynamic model of a flexible tether with $n=5$. This is justified with the assumption that the tether remains taut even for the flexible tether model, if the tension is sufficiently large~\cite{LupDAnPIICIRS13,NicNalPIWC14,TogFraPIICRA15}. When computing the control input, the direction of the \textit{taut} tether is approximated by the direction from the origin to the quadrotor, i.e., $q=\frac{x}{\norm{x}}$.

Numerical results when $T_d=10\,\mathrm{N}$ are illustrated at Figure \ref{fig:2}, where the snapshots of the controller maneuvers, and the positions of the first, the third, and the last links are presented. While the position of the last link that corresponds to the position of the quadrotor follows the desired position relatively well, there are nontrivial lateral vibrations at the first link and the third link. To reduce the vibrations, the desired tension is increased to $T_d=20\,\mathrm{N}$ at Figure \ref{fig:2_T20}. However, there still exist persistent vibrations as illustrated by Figure 4.(b) and animation.

\begin{figure}
\centerline{
	\subfigure[Snapshots of maneuver]{
		\setlength{\unitlength}{0.1\columnwidth}\scriptsize
		\begin{picture}(4.6,4.9)(0,0.6)
		\put(0,0){\includegraphics[width=0.45\columnwidth]{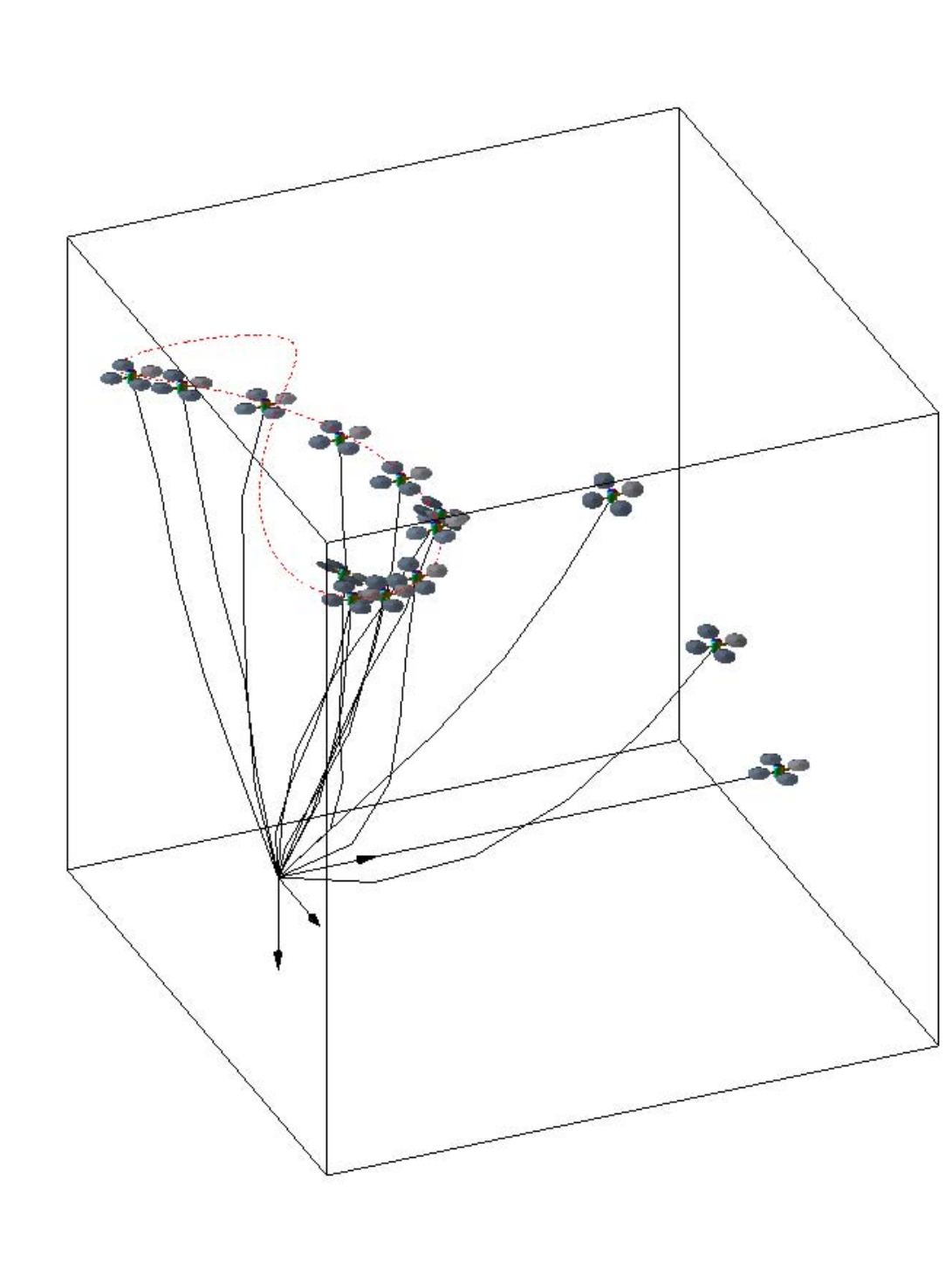}}
		\put(3.3,2){$t=0$}
		\put(2.5,4){$t=0.9$}
		\put(0.3,4.5){$t=6$}
		\end{picture}}
	\subfigure[Position of selected links $x_1,x_3,x_5=x$]{
		\includegraphics[width=0.45\columnwidth]{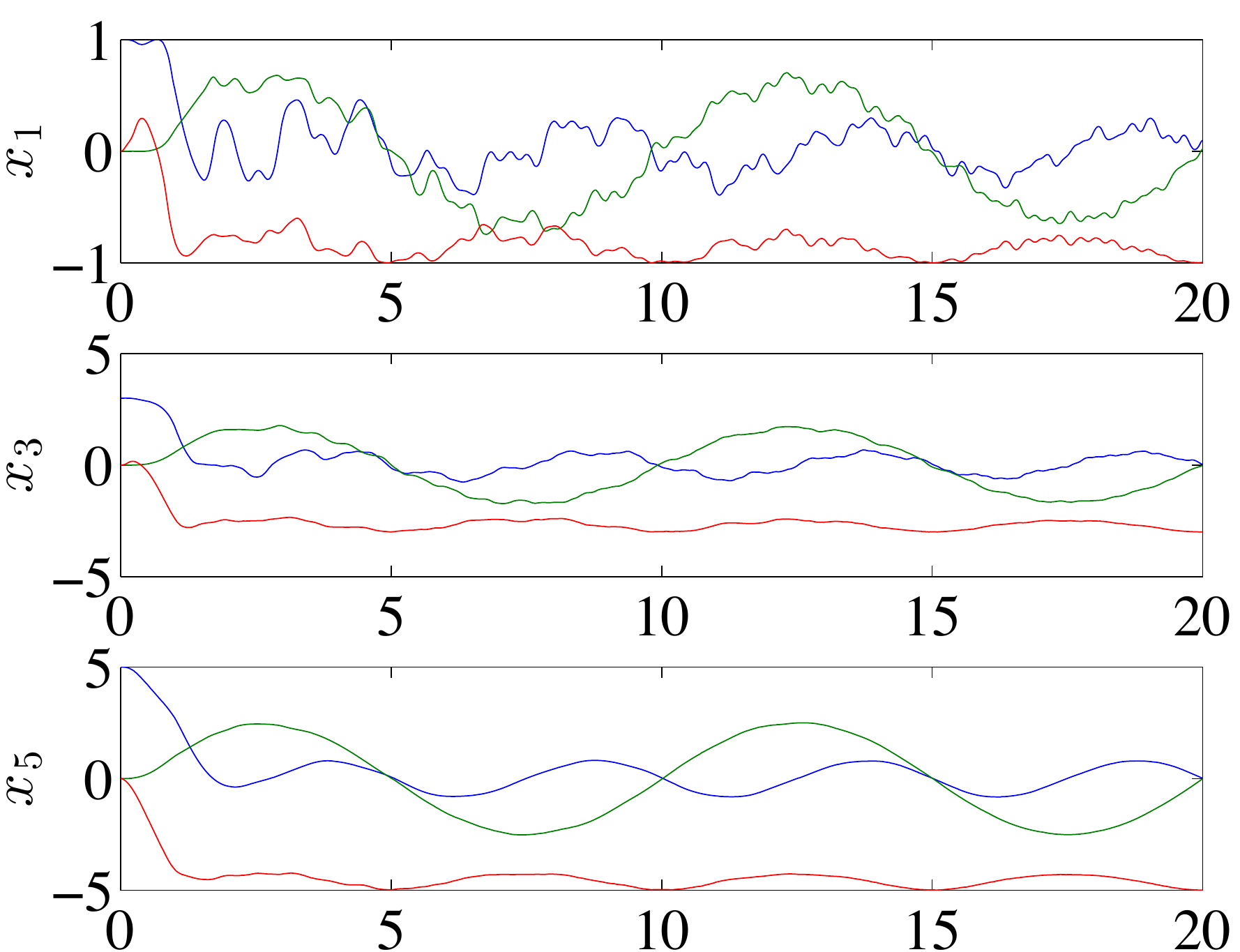}}
}
\caption{Numerical results for $n=5$, $T_d=10\,\mathrm{N}$
(animation available at \url{http://fdcl.seas.gwu.edu/CDC15_Fig3.mov})}\label{fig:2}		
\end{figure}
\begin{figure}
\centerline{
	\subfigure[Snapshots of maneuver]{
		\setlength{\unitlength}{0.1\columnwidth}\scriptsize
		\begin{picture}(4.6,4.9)(0,0.6)
		\put(0,0){\includegraphics[width=0.45\columnwidth]{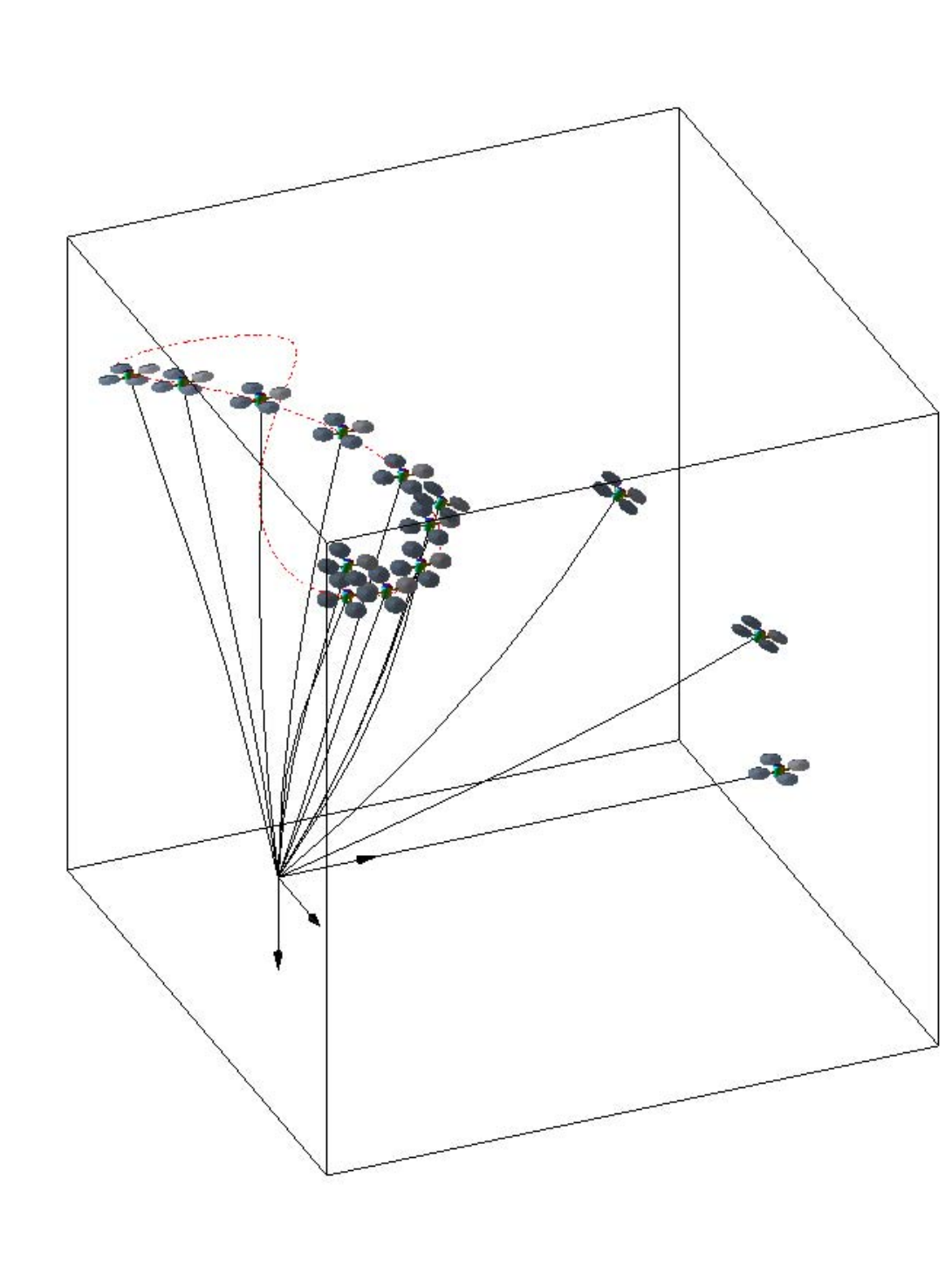}}
		\put(3.3,2){$t=0$}
		\put(2.5,4){$t=0.9$}
		\put(0.3,4.5){$t=6$}
		\end{picture}}
	\subfigure[Position of selected links $x_1,x_3,x_5=x$]{
		\includegraphics[width=0.45\columnwidth]{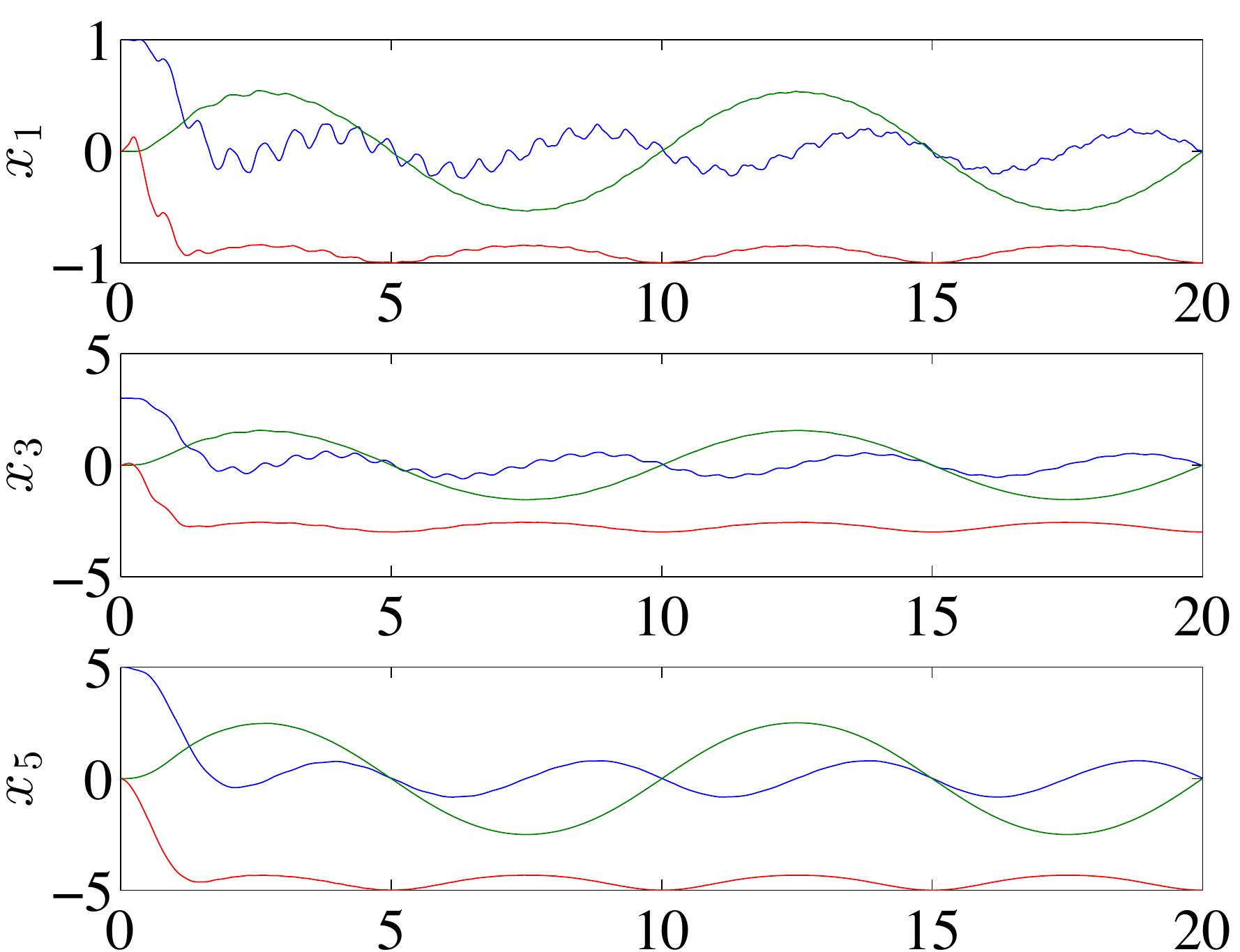}}
}
\caption{Numerical results for $n=5$, $T_d=20\,\mathrm{N}$ (animation available at \url{http://fdcl.seas.gwu.edu/CDC15_Fig4.mov})}\label{fig:2_T20}				
\end{figure}


\section{Control System Design for Flexible Tether}

The control system designed at the previous section can excite the lateral vibration of the tether when applied to the flexible tether model, and it may require increasing the tension of the tether unnecessarily large to avoid vibrations in certain cases. Motivated by these, in this section, we design another control system for the tethered quadrotor while explicitly incorporating the dynamics of a flexible tether. For simplicity, the desired configuration is selected as $q_{i_d}(t)= -e_3$, i.e., all of the links are aligned along the direction of the gravity, and the quadrotor is located directly over the pivot point at $x_d = \sum_{i=1}^n -l_i e_3$.

\subsection{Simplified Dynamic Model $(n>1)$}

Similar to the prior section, we first consider the simplified dynamic model where the total thrust $u$ can be arbitrarily selected. The proposed control system is composed of two parts: an output tracking controller to translate the quadrotor position $x$ into the vicinity of $x_d$, and a control system to asymptotically stabilize the desired configuration. 

\subsubsection{Tracking Control for Quadrotor Position}
Here, we design a control input $u$ such that the position of the quadrotor, namely $x\in\Re^3$ is translated into an intermediate point $(1-\delta)x_d$ for a constant $0<\delta <1$. Note that when $\delta$ is small, the intermediate point becomes closer to the actual desired position $x_d$. Also, $(1-\delta)x_d$ belongs to the set, 
\begin{align}
D_x = \{ x\in\Re^3 \,|\, \|x\| < \sum_{i=1}^n l_i\}, \label{eqn:Dx}
\end{align}
which is the sphere centered at the origin whose length is strictly less than the total length of tether. 

From \refeqn{qiddot_M}, $\ddot q_i$ can be written as
\begin{align}
\ddot q_i = \sum_{j=1}^n M^I_{ij}(q)\parenth{-\mathcal{G}_j(q,\dot q)-l_j\hat q_j^2 u},\label{eqn:qiddot1}
\end{align}
where $M^I_{ij}(q)\in\Re^{3\times 3}$ denotes the $(i,j)$-th block of the inverse of the $\mathcal{M}(q)$ given at \refeqn{MMq}, and $\mathcal{G}_j(q,\dot q)\in\Re^3$ is defined at \refeqn{GGi}.
Since the position of the quadrotor is $x=\sum_{i=1}^n l_i q_i$, its acceleration can be written as
\begin{align}
\ddot x 
& = -\sum_{i,j=1}^n M^I_{ij}(q)l_i\mathcal{G}_j(q,\dot q) - \parenth{\sum_{i,j=1}^n M^I_{ij}(q)l_i l_j\hat q_j^2} u\nonumber\\
&\triangleq -\mathcal{F}(q,\dot q) - \mathcal{B}(q) u,\label{eqn:xddotFT}
\end{align}
where $\mathcal{F}(q,\dot q)\in\Re^{3}$ and $\mathcal{B}(q)\in\Re^{3\times 3}$. 

\begin{assump}
The matrix $\mathcal{B}(q)$ is invertible for any configuration $q_i\in\Sph^2$ of the tether chosen such that $x=\sum_{i=1}^n l_iq_i\in D_x$.
\end{assump}
This assumption is justified by the fact that there is no restriction on the acceleration of the quadrotor in $D_x$, and therefore it can be arbitrarily changed by the control force $u$, according to Newton's second law of motion. When the quadrotor is on the boundary of $D_x$, i.e., when all of $q_i$ is identical such that the tether is taut, the control input cannot generate any acceleration along $q_i$, due to the constraints that the total length of the tether is fixed. This can also be observed from the fact that when all of $q_i$ are identical, the matrix $\mathcal{B}(q)$ has a null space spanned by $q_i$, i.e., the component of the control force $u$ parallel to $q_i$ does not affects the acceleration $\ddot x$ when all of $q_i$ are identical.

Define a desired trajectory $y_d(t)$ as
\begin{align}
y_d(t) = x(0)e^{-\gamma t} + (1-\delta) (1-e^{-\gamma t})x_d ,
\end{align}
for $\gamma >0$. This satisfies $y_d(t)=x(0)$ and $\lim_{t \rightarrow \infty} = (1-\delta) x_d$. Also, $y_d(t)\in D_x$ for any $t\geq 0$, if $x(0)\in D_x$ due to the convexity of $D_x$. In other words, $y_d(t)$ corresponds to a parameterized line connecting the initial point $x(0)$ and the intermediate point $(1-\delta)x_d$.

Let the position tracking error be $e_x=x-y_d\in\Re^3$. The control input is designed according to output feedback linearization as
\begin{align}
u = -\mathcal{B}^{-1}(q)\braces{\mathcal{F}(q,\dot q)-k_x e_x - k_{\dot x} \dot e_x + \ddot y_d},\label{eqn:uFT}
\end{align}
for positive gains $k_x,k_{\dot x}$. 

\begin{prop}
Consider the simplified tethered quadrotor model described by \refeqn{qiddot}, where $n>1$ and $u$ can be arbitrarily selected, with Assumption 1. If the initial position $x(0)$ belongs to the domain $D_x$ given at \refeqn{Dx}, then there exist controller gains $k_x,k_{\dot x}$ such that $x(t)\in D_x$ for all $t\geq 0$, and $(e_x,\dot e_x)=(0,0)$ is exponentially stable. 
\end{prop}
\begin{proof}
See Appendix D.
\end{proof}

\subsubsection{Stabilization for Tether} The above tracking control system guarantees that the quadrotor is translated into the intermediate point $(1-\delta)x_d$ that is arbitrarily close to the actual desired point $x_d$. But, it does not guarantee that the motion of the tether is asymptotically damped out. Therefore, we introduce another control system that stabilizes the tether as well as the quadrotor. Due to the high degrees of underactuation, it is designed based on the linearized dynamics. 

At the desired equilibrium configuration, we have $q_i = -e_3$, $\omega_i=0$ and $u_d = -m_T g e_3$, where $m_T=m+\sum_{i=1}^n m_i$ denotes the total mass of the links and the quadrotor. An intrinsic formulation of the linearized equations on $\Sph^2$ has been developed in~\cite{LeeLeoPICDC12}. According to it, the variations from the equilibrium can be written as
\begin{align}
q_i^\epsilon = -\exp(\epsilon \hat \xi_i)e_3,\quad \omega_i^\epsilon = \epsilon\delta\omega_i,\label{eqn:qieps}
\end{align}
where $\xi_i,\delta\omega_i\in\Re^3$ with $\xi_i\cdot e_3=0$ and $\delta\omega_i\cdot e_3=0$. This yields the following infinitesimal variation $\delta q_i = -\xi_i\times e_3$.

Substituting this into \refeqn{dotqi}, \refeqn{widot} and ignoring the higher order terms, the linearized equations can be written as
\begin{align}
\mathbf{M} \ddot{\mathbf{x}} + \mathbf{G}\mathbf{x} = \mathbf{B} \delta u,\label{eqn:LEOM}
\end{align}
where $\mathbf{x}=[C^T\xi_1; C^T\xi_2;\ldots; C^T\xi_n]\in\Re^{2n}$ corresponds to the state vector of the linearized dynamics with $C=[e_1,e_2]\in\Re^{3\times 2}$, $e_1=[1,0,0]^T$, and $e_2=[0,1,0]\in\Re^3$. The control input for the linearized dynamics is $\delta u = u-u_d\in\Re^3$, and the matrices $\mathbf{M},\mathbf{G}\in\Re^{2n\times 2n}$, $\mathbf{B}\in\Re^{2n\times 3}$ are defined as (see Appendix E)
\begin{align*}
\mathbf{M}&=\begin{bmatrix}%
	M_{11}I_{2\times 2} & M_{12} I_{2\times 2} & \cdots & M_{1n}I_{2\times 2}\\%
	M_{21} I_{2\times 2} & M_{22} I_{2\times 2} & \cdots & M_{2n}I_{2\times 2}\\%
	\vdots & \vdots & & \vdots\\
	M_{n1}I_{2\times 2} & M_{n2}I_{2\times 2} & \cdots & M_{nn} I_{2\times 2}
    \end{bmatrix},\\
\mathbf{G}&=\mathrm{diag}[(m_T-M_{g_1})gl_1I_{2\times 2},\ldots, (m_T-M_{g_n})gl_nI_{2\times 2}],\\
\mathbf{B}&=    \begin{bmatrix}
	-l_1 C^T \hat e_3\\
	-l_2 C^T \hat e_3\\
	\vdots\\
	-l_n C^T \hat e_3\\
    \end{bmatrix}.
\end{align*}

The control input is designed as
\begin{align}
u = - \mathbf{K_x}\mathbf{x}- \mathbf{K_{\dot x}}\dot{\mathbf{x}} -m_T ge_3,\label{eqn:uLC}
\end{align}
where the controller gains,  $\mathbf{K_x},\mathbf{K_{\dot x}}\in\Re^{3\times 2n}$ are selected such that the linearized dynamics \refeqn{LEOM} becomes Hurwitz. This provides asymptotic stability of the desired equilibrium according to the Lyapunov indirect method.

In short, the tracking control for the quadrotor position, \refeqn{uFT} is engaged first such that the quadrotor becomes sufficiently close to the desired equilibrium. Then, the linear control \refeqn{uLC} is applied to asymptotically stabilize both the quadrotor and the tether.

\subsection{Full Dynamic Model $(n>1)$}

The procedures to extend the above control systems into the full dynamic model, that incorporate the attitude dynamics of the quadrotor, are identical to \refeqn{b3i}-\refeqn{Mi} described at Section III-C.

\subsection{Numerical Example}

The properties of the quadrotor and the tether are identical to Section \ref{sec:NE1}. The initial conditions are chosen such that the quadrotor is at $x=[2.46,0,-0.43]^T\in\Re^3$, and the tether is hanging while minimizing the gravitational potential. The intermediate position is chosen with $\delta=0.01$, $\gamma=1$, and the switching occurs at $t=3$. The corresponding numerical results are illustrated at Figure \ref{fig:3}, where it is shown that the quadrotor is translated to the desired position asymptotically. In contrast to Figures 3.(b) and 4.(b), the vibration of the tether is effectively eliminated at Figure 5.(d) and the presented animation. 

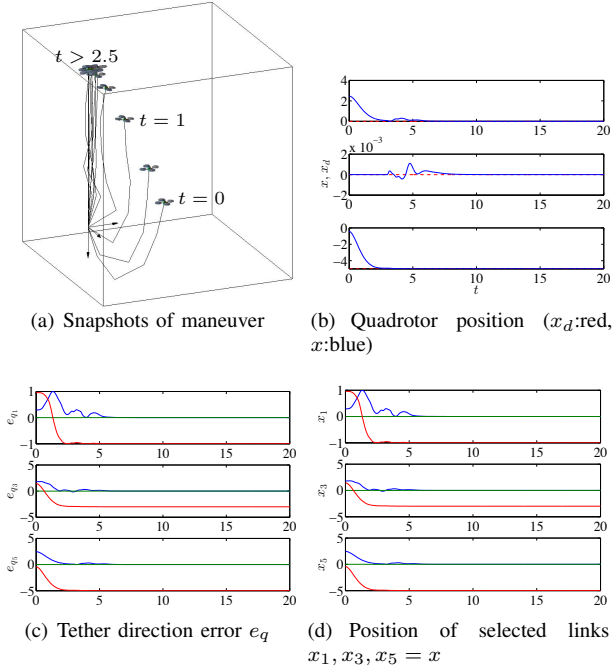
\begin{figure}
\centerline{
	\subfigure[Snapshots of maneuver]{
		\setlength{\unitlength}{0.1\columnwidth}\scriptsize
		\begin{picture}(4.6,4.9)(0,0.6)
		\put(0,0){\includegraphics[width=0.45\columnwidth]{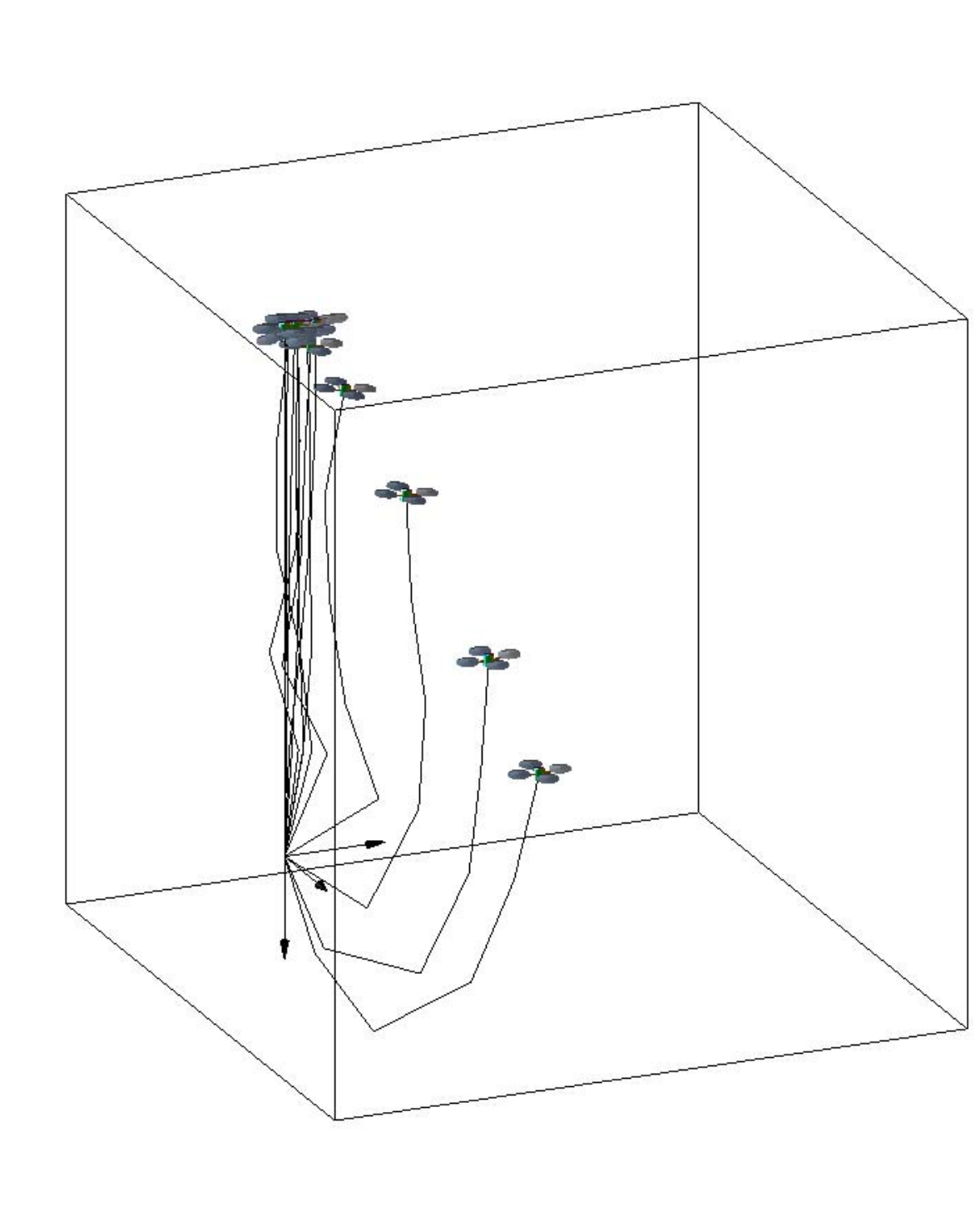}}
		\put(2.7,2){$t=0$}
		\put(2.1,3.2){$t=1$}
		\put(0.8,4.2){$t>2.5$}
		\end{picture}}
	\subfigure[Quadrotor position ($x_d$:red, $x$:blue)]{
		\includegraphics[width=0.45\columnwidth]{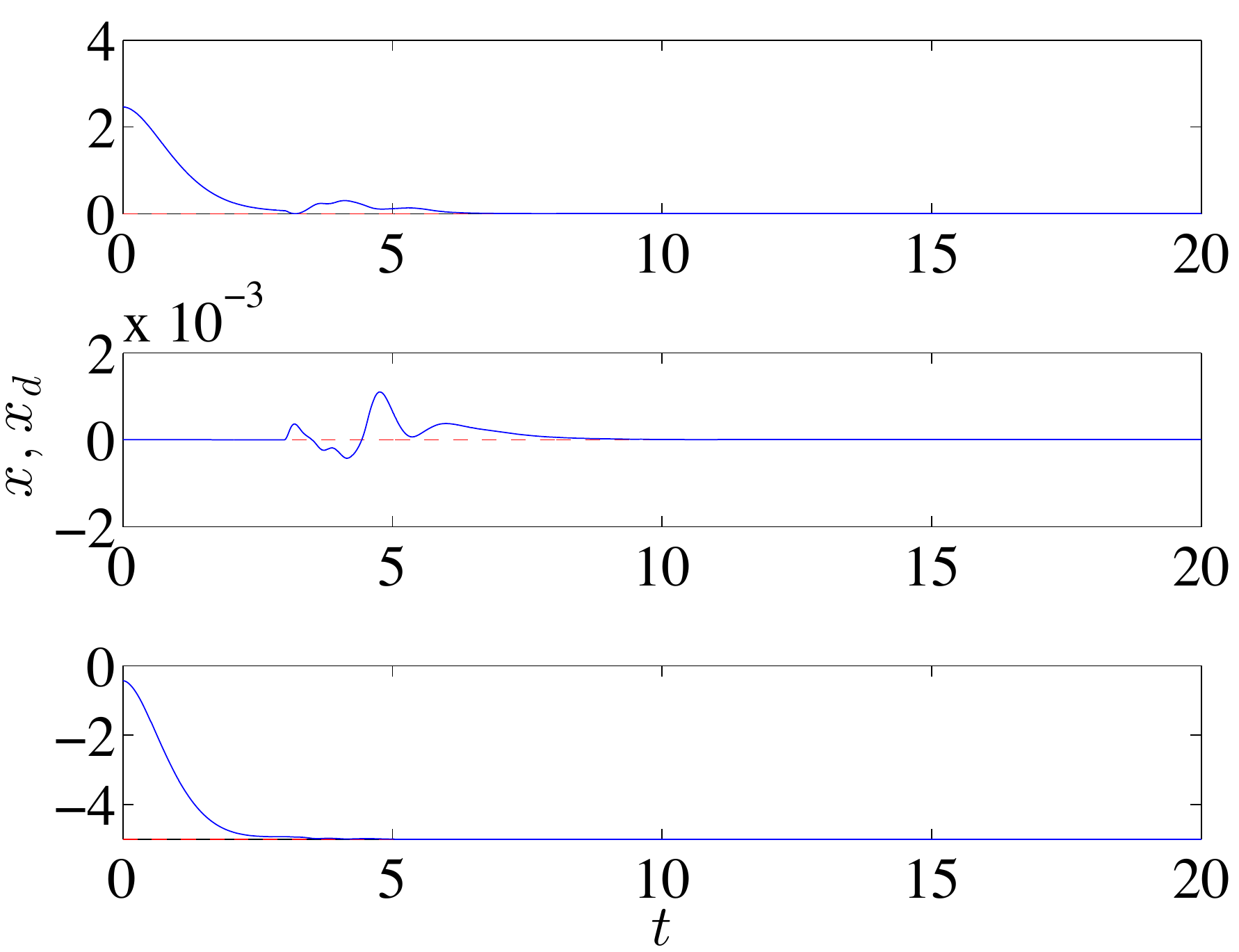}}
}		
\centerline{
	\subfigure[Tether direction error $e_q$]{
		\includegraphics[width=0.45\columnwidth]{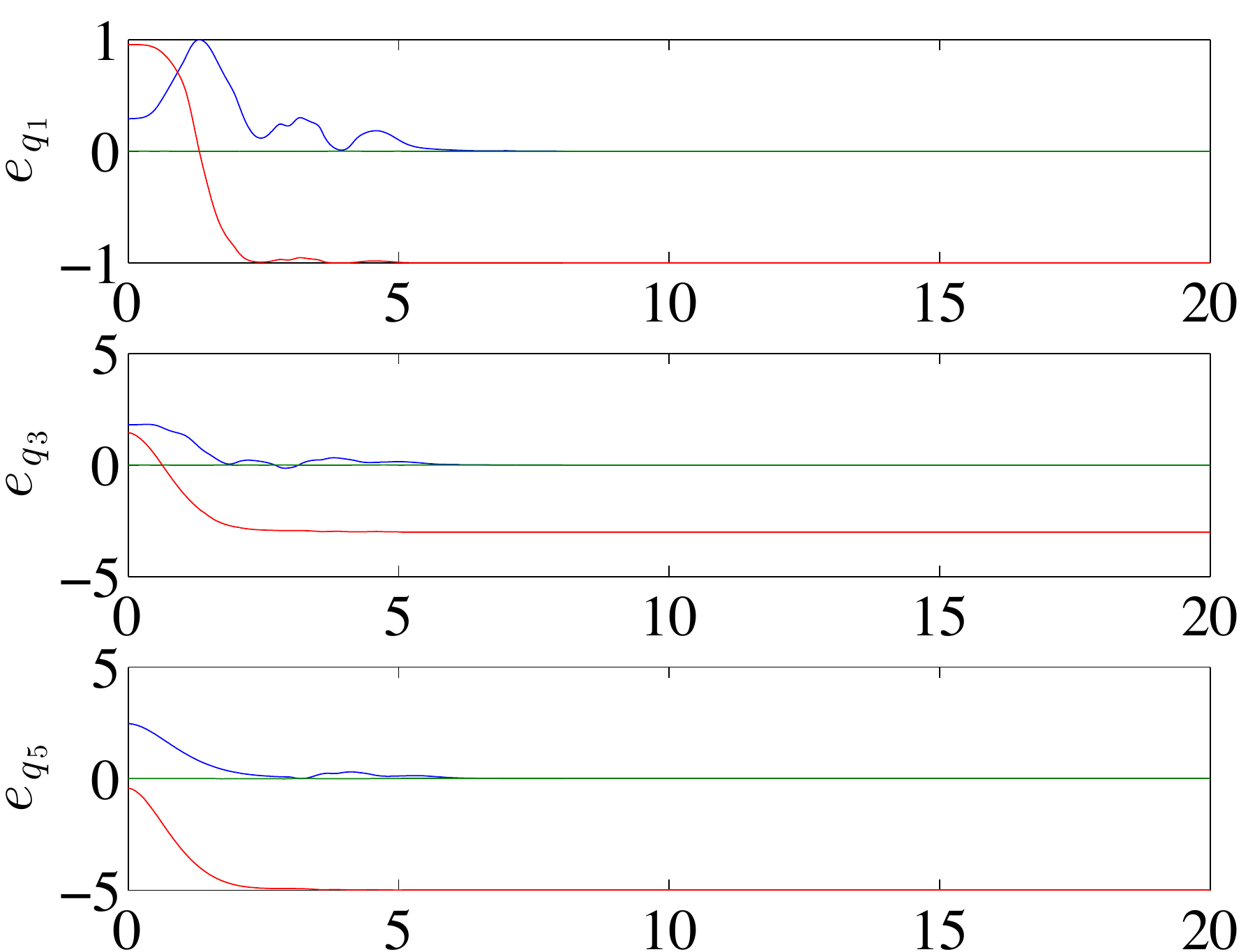}}
	\subfigure[Position of selected links $x_1,x_3,x_5=x$]{
		\includegraphics[width=0.45\columnwidth]{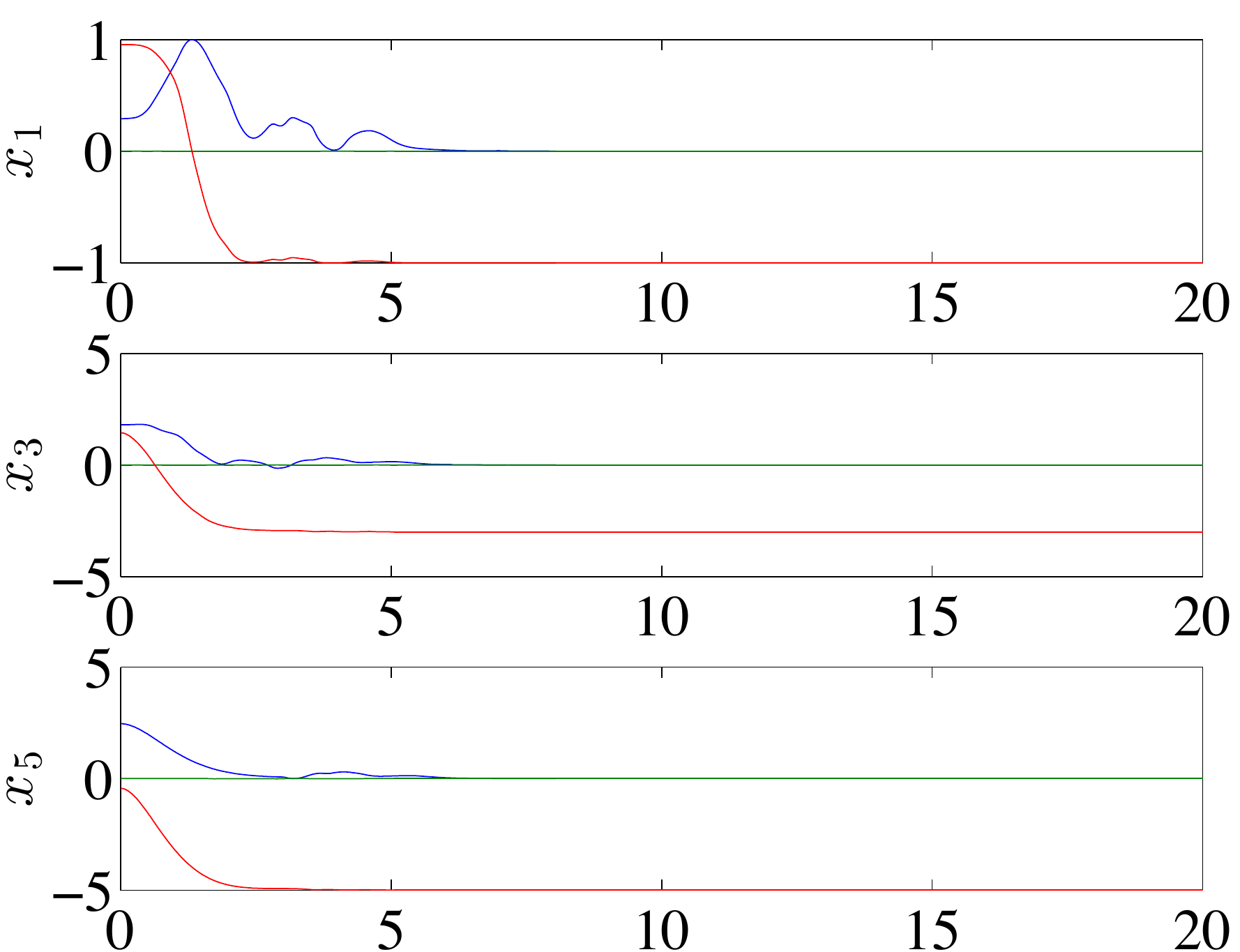}}
}
\caption{Numerical results for $n=5$ (animation available at \url{http://fdcl.seas.gwu.edu/CDC15_Fig5.mov})}\label{fig:3}
\end{figure}

In short, the control system presented in this section is developed for the flexible cable model when $n>1$ at the cost of increased complexity. As illustrated by numerical examples, the undesired lateral vibrations of the tether, observed at Section III are eliminated. The development of the dynamic model and the control system design for the tethered quadrotor with flexible tether have ben unprecedented. While the proposed approach is developed for a stabilization problem where the desired link direction $q_{i_d}$ is fixed, but it is readily generalized to the tracking problems.

\appendix

\subsection{Euler--Lagrange equations}

Here, we develop the Euler--Lagrange equations for the Lagrangian given by \refeqn{T} and \refeqn{U}. The Lagrangian is independent of $R$. The derivatives of the Lagrangian with respect to $\dot q_i,\Omega,q_i,\Omega$ are given by
\begin{gather}
\D_{\dot q_i} \mathcal{L} = \sum_{j=1}^n M_{ij}\dot q_j,\quad
\D_{q_i} \mathcal{L} = M_{g_i}l_i g e_3,\quad
\D_{\Omega} \mathcal{L} = J\Omega.\label{eqn:dL_qi}
\end{gather}
Substituting $\delta R=R\hat\eta$ into the attitude kinematic equations \refeqn{dotR} and rearranging, the variation of the angular velocity can be written as $\delta\Omega=\dot\eta + \Omega\times\eta$~\cite{Lee08}. For the variation model of $q_i$ given at \refeqn{delqi}, we have $\delta q_i =\xi_i\times q_i$ and $\dot \xi_i = \dot\xi_i\times q_i + \xi_i\times \dot q_i$~\cite{LeeLeoIJNME09}. 

Let $\mathfrak{G}=\int_{t_0}^{t_f}\mathcal{L}\,dt$ be the action integral. Using the above expression for the variations, and integrating by parts, the variation of the action integral can be written as
\begin{align*}
\delta\mathfrak{G} & = \int_{t_0}^{t_f} 
\braces{-\frac{d}{dt}\D_{\Omega} \mathcal{L} -\Omega\times \D_{\Omega} \mathcal{L}}\cdot\eta\\
&\quad +\sum_{i=1}^n \braces{-q_i\times \frac{d}{dt}\D_{\dot q_i}\mathcal{L}+q_i\times \D_{q_i} \mathcal{L}}\cdot\xi_i \,dt.
\end{align*}

The total thrust of the quadrotor with respect to the inertial frame is given by $u=-fRe_3\in\Re^3$ and the total moment of the quadrotor is $M\in\Re^3$ with respect to the body-fixed frame. The corresponding virtual work can be written as
\begin{align*}
\delta\mathcal{W} = \int_{t_0}^{t_f} u\cdot \sum_{i=1}^n l_i (\xi_i\times q_i) + M\cdot\eta \,dt. 
\end{align*}

According to the Lagrange-d'Alembert principle, we have $\delta\mathfrak{G}=-\delta\mathcal{W}$ to obtain
\begin{gather*}
\frac{d}{dt}\D_{\Omega} \mathcal{L} +\Omega\times \D_{\Omega} \mathcal{L}=M,\\
-q_i\times \frac{d}{dt}\D_{\dot q_i}\mathcal{L}+q_i\times \D_{q_i} \mathcal{L} = -l_i\hat q_i u.
\end{gather*}
The first equation yields \refeqn{Wdot}. Multiplying both sides of the second equation with $\hat q_i$, and substituting \refeqn{dL_qi}, \refeqn{dL_qi},
\begin{align}
-\hat q_i^2\sum_{j=1}^n M_{ij}\ddot q_j +\hat q_i^2M_{g_i}l_i g e_3 = -l_i\hat q_i^2 u.\label{eqn:tmp0}
\end{align}
Since $q_i\cdot \dot q_i=0$, it follows that $q_i\cdot\ddot q_i +\|\dot q_i\|^2=0$. Thus,
\begin{align*}
-\hat q_i^2 \ddot q_i = (I_{3\times 3}-q_iq_i^T)\ddot q_i = \ddot q_i +\|\dot q_i\|^2 q_i.
\end{align*}
Using this identity, \refeqn{tmp0} can be rewritten as
\begin{align*}
M_{ii}\ddot q_{i} -\hat q_i^2\sum_{\substack{j=1\\j\neq i}}^n M_{ij}\ddot q_j + M_{ii}\|\dot q_i\|^2 q_i +\hat q_i^2 M_{g_i}l_i g e_3 = -\hat q_i^2l_i u_i,
\end{align*}
which corresponds to \refeqn{qiddot}. Rearranging \refeqn{qiddot} with the fact that $\ddot q_i = -\hat q_i\dot\omega_i - \|\omega_i\|^2 q_i$ and $\hat q_i\ddot q_i = -\hat q_i^2\dot \omega_i =\dot\omega_i$~\cite{LeeLeoIJNME09}, we obtain \refeqn{widot}.

\subsection{Proof of Proposition 1}

Define an error function, $\Psi(q)=1-q\cdot q_d$. For a positive constant $\psi_q<2$, define the following open domain containing the zero equilibrium, $D_q=\{(q,\omega)\in\Sph^2\times \Re^3\,|\, \Psi_q < \psi_q$. Then, it is shown that
\begin{align*}
\frac{1}{2}\norm{e_q}^2\leq \Psi_q \leq \frac{1}{2-\psi_q}\norm{e_q}^2,
\end{align*}
where the upper bound is satisfied for any $q\in D_q$~\cite{Wu12}. Define a Lyapunov function as
\begin{align*}
\mathcal{V}_q = \frac{1}{2}\|e_\omega\|^2 + c_q e_\omega \cdot e_q + k_q \Psi_q,
\end{align*}
which is bounded as
\begin{align*}
z_q^T \underline{P}_q z_q \leq \mathcal{V}_q \leq z_q^T \overline{P}_q z_q,
\end{align*}
where $z_q=[\|e_q\|,\|e_\omega\|]\in\Re^2$, and the symmetric matrices $\underline{P}_q,\overline{P}_q\in\Re^{2\times 2}$ are defined as
\begin{align*}
\underline{P}_q = \frac{1}{2}\begin{bmatrix} 2k_q & -c_q \\ -c_q & 1\end{bmatrix},\quad
\overline{P}_q = \frac{1}{2}\begin{bmatrix} \frac{2k_q}{2-\psi_q} & c_q \\ c_q & 1\end{bmatrix}.
\end{align*}
The time-derivative of $\mathcal{V}_q$ along \refeqn{dotwu} can be written as
\begin{align*}
\dot{\mathcal{V}}_q & = -z_q^T W_q z_q,
\end{align*}
where the matrix $W_q\in\Re^{2\times 2}$ is defined as (\cite{Wu12})
\begin{align*}
W_q =\begin{bmatrix} c_q k_q & -\frac{c_qk_\omega}{2} \\
-\frac{c_qk_\omega}{2} & k_\omega -c_q\end{bmatrix}.
\end{align*}
If the constant $c_q$ is sufficiently small, all of the matrices $\underline{P}_q,\overline{P}_q, W_q$ are positive-definite, which follows that the zero equilibrium of the tracking errors is exponentially stable. Substituting \refeqn{uprl1} into \refeqn{T} yields $T=T_d$.

\subsection{Proof of Proposition 2}

The proof is based on the singular perturbation theory~\cite{Kha96}, i.e., if the attitude dynamics is sufficiently fast, the stability properties of the \textit{reduced system} summarized by Proposition 1 holds. More explicitly, the \textit{boundary-layer} system corresponds to the attitude dynamics of the quadrotor, and the attitude tracking errors exponentially converge to zero at the rate proportional to $\frac{1}{\epsilon}$~\cite[Proposition 2]{LeeSrePICDC13}. The \textit{reduced system} represents the dynamics of the link when $R=R_c$, and from \refeqn{b3i}, \refeqn{Rc}, and \refeqn{fi},
\begin{align*}
-fRe_3 = (u\cdot R_c e_3) R_c e_3 = (u\cdot  \frac{u}{\norm{u}})\frac{u}{\norm{u}}=u.
\end{align*}
Therefore, the reduced system corresponds to the simplified dynamic model analyzed at Proposition 1. Then, according to Tikhonov's theorem~\cite[Thm 9.3]{Kha96}, there exists $\epsilon^\star >0$ such that for all $\epsilon<\epsilon^\star$, the origin of the full dynamics model is exponentially stable.

\subsection{Proof of Proposition 3}

We first show that there exist controller gains such that $x(t)\in D_x$ for all $t\geq 0$. Substituting \refeqn{uFT} into \refeqn{xddotFT}, we obtain the following linear error dynamics.
\begin{align}
\ddot e_x = -k_x e_x -k_{\dot x} \dot e_x.\label{eqn:exddot}
\end{align}
Define
\begin{align*}
\mathcal{U} = \frac{1}{2}\|\dot e_x\|^2 + \frac{k_x}{2} \|e_x\|^2.
\end{align*}
Its time-derivative along the solutions of \refeqn{exddot} is given by $\dot{\mathcal{U}}=-k_{\dot x} \|\dot e_x\|^2$, which follows 
\begin{align*}
k_x \|e_x(t)\|^2 \leq \mathcal{U}(t) \leq \mathcal{U}(0)=\frac{1}{2}\|\dot e_x(0)\|^2. 
\end{align*} 
Therefore, $\|e_x(t)\|\leq \frac{1}{\sqrt{2k_x}}\|\dot e_x(0)\|$ for any $t\geq 0$. Since $\|x\|-\|y_d\|\leq \|x-y_d\| = \|e_x\|$, 
\begin{align*}
\|x(t)\| \leq \|y_d(t)\| + \frac{1}{\sqrt{2k_x}}\|\dot e_x(0)\|.
\end{align*}
As $y_d(t)$ lies in $D_x$ always, $\|y_d(t)\| < \sum_{i=1}^l l_i$, and therefore, if $k_x$ is sufficiently small, the above inequality guarantees $\|x(t)\| < \sum_{i=1}^l l_i$ as well. Therefore, $x(t)\in D_x$ for all $t\geq 0$. 

According to Assumption 1, the control input \refeqn{uFT} is well-defined, and it is straightforward to show the exponential stability of the linear error dynamics given by \refeqn{exddot}.


\subsection{Linearization}


The perturbation model given at \refeqn{qieps} yields $\delta q_i=-\xi_i\times e_3$. Substituting it into \refeqn{dotqi}, $\delta\dot q_i$ is given by
\begin{align*}
\delta \dot q_i = \dot\xi_i \times -e_3 =\delta\omega_i \times -e_3 + 0\times (\xi_i\times -e_3)=\delta\omega_i \times -e_3.
\end{align*}
Since both sides of the above equation is perpendicular to $e_3$, this is equivalent to $e_3\times(\dot\xi_i\times e_3) = e_3\times(\delta\omega_i\times e_3)$, which yields
\begin{gather*}
\dot \xi - (e_3\cdot\dot\xi) e_3 = \delta\omega_i -(e_3\cdot\delta\omega_i)e_3.
\end{gather*}
Since $\xi_i\cdot e_3 =0$, we have $\dot\xi\cdot e_3=0$. Also, $e_3\cdot\delta\omega_i=0$ from the constraint. Substituting these to above, we obtain the linearized equation for the kinematics equation:
\begin{align}
\dot\xi_i = \delta\omega_i.\label{eqn:dotxii}
\end{align}
Substituting these into \refeqn{widot}, and ignoring the higher order terms,
\begin{align*}
M_{ii}\delta\dot\omega_i - \sum_{\substack{j=1\\j\neq i}}^nM_{ij} \hat {e}_3^2 \delta\dot\omega_j
+M_{g_i}l_ig \hat e_3^2 \xi_i = -l_i \hat e_3\delta u + l_im_T g\hat e_3^2 \xi_i.
\end{align*}
Let $C=[e_1, e_2]\in\Re^{3\times 2}$. Multiplying the both side of the above equation by $C^T$, and rearranging it with the facts that $C^T\hat e_3^2=-C^T$, $\hat e_3^2 = \mathrm{diag}[-1,-1,0]$, $C^T\hat e_3^2 C=-I_2$ and $\hat e_3 C C^T = \hat e_3$, we obtain \refeqn{LEOM}.

\bibliography{/Users/tylee/Documents/BibMaster}
\bibliographystyle{IEEEtran}

\end{document}